\documentclass[a4paper,12pt]{article}
\usepackage[a4paper]{geometry}
\usepackage[OT2,T1]{fontenc}
\usepackage{lipsum}\usepackage[english]{babel}
\usepackage{amssymb,amsmath,amsthm,amscd,amsfonts}
\usepackage{mathrsfs}
\usepackage{tikz-qtree}
\usetikzlibrary{shapes,arrows.meta}
\usepackage{subfig}
\usepackage{tabularx}
\usepackage{calc} 
\usepackage{graphicx}
\usepackage[nottoc]{tocbibind}
\usepackage{makeidx}
\usepackage{authblk}
\makeindex
\usepackage{longtable}
\usepackage{enumerate}
\usepackage{array}
\usepackage{mathrsfs}
\usepackage{version}
\usepackage{lscape}
\usepackage{epigraph}
\usepackage[all]{xy}
\usepackage{datetime}
\usepackage{multirow}
\usepackage{pgf,tikz}
\usepackage{etoolbox}
\usepackage{comment}
\usepackage{url}
\usepackage[font=scriptsize]{caption}
\apptocmd{\thebibliography}{\raggedright}{}{}

\newtheoremstyle{one}
{11pt}
{11pt}
{\it}
{}
{\bf}
{.}
{1mm}
{}

\newtheoremstyle{two}
{11pt}
{11pt}
{}
{}
{\bf}
{.}
{1mm}
{}

\theoremstyle{one}
\newtheorem{theorem}{Theorem}[section]
\newtheorem{lemma}[theorem]{Lemma}

\theoremstyle{two}

\newtheorem{example}[theorem]{Example}

\newtheorem{remark}[theorem]{Remark}

\def\Z{\mathbb{Z}}

\def\Q{\mathbb{Q}}

\def\E{\mathscr{E}}
\def\F{\mathscr{F}}

\title{Moderate rank jumps on rational elliptic surfaces via construction of conics}

\author{Julie Desjardins}
\date{}

\begin{document}
\maketitle

\setcounter{tocdepth}{2}

\begin{abstract}
    This note is devoted to studying certain families of elliptic surfaces with infinitely many fibers with rank at least $3$ or $4$ revisiting and combining ideas from of Gary Walsh, Salgado and Loughran, and the author.
\end{abstract}

\section{Introduction}

In this note, we describe a rank jump phenomenon on large classes of rational elliptic surfaces using a combination of ideas of Salgado-Loughran \cite{LS20} 
and Walsh \cite{Walsh1},\cite{Walsh2}. 
In particular, it re-explains and extends the following example:

\begin{example}(Revisiting of \cite[Theorem 1]{Walsh2})\label{ex:Walsh1}

Let $a,b\in\Q-\{0\}$, and let $\E$ be the elliptic surface given by $$\mathcal{G}:y^2=x(x-a)(x-b)+t^2,$$ then the set of fibers with rank 3 is infinite.

This can be proven by the fact that by \cite[Theorem 1]{Walsh1} $\mathcal{G}$ has generic rank $2$. By the additional restriction to the fibers such that $t=m^3$ we give rise to the elliptic surface $\mathcal{F}$ given by the equation $$\F:y^2=x(x-a)(x-b)+m^6.$$ This surface $\mathcal{F}$ is a triple cover of $\mathcal{G}:y^2=x(x-a)(x-b)+m^2$, so the sections of $\mathcal{F}$ restricted to $\mathcal{G}$. The base change by the conic $Y^2=(a+b)X^2+ab$ has a third independent non-torsion section $(-m^2,mn)$ by \cite[Thm 2.4]{LS20} on the surface resulting from a base change of $\mathcal{G}$ by the conic.

Observe that $\F$ has a singular fiber of type $IV^*$ at infinity. By Shioda-Tate formula as well as \cite{OS} the generic rank of $\E$ is at most $2$. Moreover, $\E$ satisfies the hypothesis iii) of \cite[Proposition 3.2]{DCS} and has rank 2, and thus we get by \cite[Theorem 1.1]{DCS} that the set of fibres of $\E$ with rank 5 or more is non-thin, a stronger statement than Example \ref{ex:Walsh1}.
\end{example}

The base change in this example enlightens the construction of the sections on the surface $\mathcal{F}$: while the conic $C:\{x=A, y^2=(a+b)t^2+ab\}$ is a bisection of $\E$, the curve $\tilde{C}$ obtained by change of variable $t=m^3$ is not anymore a bisection. Indeed, it intersection the fiber at $t=0$ of $\F$ with multiplicity $6$. However, the twist of one of the irreducible components of $\tilde{C}$ by $m$ is the section $(-m^2,mn)$.

Observe furthermore that assuming that Chowla's conjecture holds on the polynomial made of the multiplicative places and moreover assuming the parity conjecture, we can deduce from \cite{Desjardins1} that since $\F$ in Example \ref{ex:Walsh1} is non-isotrivial, then on the surfaces $\E$ and $\mathcal{G}$ the set of fibers with rank $4$ is infinite.
\\

In this note, we construct a similar family of elliptic surface with two sections and a conic, providing explicit equations. The key is to construct some genus 0 curves that are bisections or components of bisections on an elliptic surface $\E$. 
Observe that $\mathcal{F}$ and the surfaces in Theorem \ref{thm:introresult} 
generically have no reducible fibers, and thus the contraction of its zero-section gives a del Pezzo surface of degree 1. On those surfaces, the generic rank of $\E$ over $\Q[t]$ could be as big as 8. Moreover the case of the surfaces $\mathcal{D}$ in Theorem 
\ref{thm:introresult} is not covered by \cite{DCS}. They are however covered by \cite{LS20} in a less explicit way as we explain in Remark \ref{rem:LS}.

\begin{theorem}\label{thm:introresult}

Let $s,w\in\Q$ be distinct numbers such that the conic $Y^2+(s+w)T^2=sw$ is solvable. Let $(t,m)$ be a solution. Then for a given $v\in\Q$, there is only finitely many exceptions to the fact that the elliptic curve given by the Weierstrass equation
    $$\mathcal{D}:y^2=x^3-(3vt+s+w) x^2+ 
(3v^2t^2 -2v(s+w)t+sw)x
+ t^6 - v^3t^3 -v^2(s+w)t^2+ swvt$$
has rank at least 3, with the three following independant non-torsion points $[-t^2+vt,tm]$, $[vt+s,t^3]$, $[vt+w,t^3]$.
\end{theorem}

\begin{remark}
    Walsh's example \ref{ex:Walsh1} is included in our Theorem \ref{thm:introresult}: put $v=0$.
\end{remark}

We chose to present the subcase of Theorem \ref{thm:encoreplusmain} for sake of simplicity and to lighten the introduction from technical equations, but the same construction of the three points apply if we allow the coefficient of the term in $t^2$ to be non-zero by adding convienient additional parameters, as in Theorem \ref{thm:encoreplusmain} presented in section \ref{sec:proofs}. Theorem \ref{thm:introresult} is thus simply Theorem \ref{thm:encoreplusmain} where those additionnal parameters are all $0$.

\begin{example}
Let fix $w=-1,s=-4$ and consider the surface given by Theorem \ref{thm:introresult}: $$\mathcal{D}:y^2=x^3+\left(-\frac{9}{4}t^2+5\right)x^2+\left(\frac{45}{64}t^4 - \frac{15}{4}t^2+4\right)x+
\frac{25}{16}t^6 + \frac{45}{64}t^4-\frac{3}{2}t^2.$$

This surface comes from the surface $\mathcal{F}:y^2=x^3+(-9/4n + 5)x^2+\left(\frac{45}{64}n^2-\frac{15}{4}n+4\right)x+
\frac{25}{16}n^3 + \frac{45}{64}n^2 - \frac{3}{2}n$ by the base change $n=t^2$. We retrieve from Theorem \ref{thm:introresult} three bisections of genus $0$ on $\mathcal{F}$ given by
$$B_1=\left[\frac{15}{8}n - 4,\frac{5}{4}n\sqrt{n}\right],B_2=\left[-1,\frac{5}{4}n\sqrt{n}\right],B_3=\left[-\frac{5}{8}n,\sqrt{5n^2 - 4n}\right]$$ Thus $\mathcal{F}$ also has infinitely many fibers of rank at least $3$.
\end{example}

\begin{remark}
In general, for any choice of $s,w\in\Z$, $S_1$ and $S_2$ are sections of $\mathcal{D}$ and that surface has generic rank at least 2. However, the bisection $B_1:=[at^2+bt,t\sqrt{(-s-w)t^2-sw}]$ will never be defined over the rational for any value of $s,w$ because of the assumption that $s$ is equal $w$ modulo squares. The point $B_1(t)$ is defined over $\Q$ if and only if $\sqrt{(-s-w)^2-sw}$ is rational. In other words, the curve $B_1$ has rational points if and only if the conic $C$ has rational points as well. 
\end{remark}

There are more rational elliptic surfaces then what is covered by Theorem \ref{thm:introresult} or Theorem \ref{thm:encoreplusmain} that have two independant sections coming from a pair of torsion points -- we can prove in the same way Theorem \ref{thm:rank2withconic} which is also a main result of the present paper but whose (even more technical) statement we relegate to the final section
. A surface fulfilling the hypothesis of Theorem \ref{thm:rank2withconic} is the following Example \ref{ex:rank2withconicintro}. Note that there are even more surfaces with a similar property as Walsh's example, as illustrated in further Example \ref{ex:jump2} in Section \ref{sec:examples} not contained in any of our theorem, but we thought an explicit statement containing all of them would be both too technical and not really interesting. This paper should be viewed as a recipe to create such examples.

\begin{example}\label{ex:rank2withconicintro}
Let $\mathcal{H}$ be the elliptic surface given by the equation $$\mathcal{H}:y^2=x^3+\left(-\frac{79}{18}t^2 - 11t + 10\right)x^2+
\left(\frac{502}{81}t^4 + \frac{287}{9}t^3 - \frac{179}{9}t^2 - 20t + 9\right)x$$$$+
\left(-\frac{901}{324}t^6 - \frac{188}{9}t^5 + \frac{835}{81}t^4 + \frac{269}{9}t^3 - \frac{15}{2}t^2 - 9t\right)$$
Then the generic rank is 2, with generating sections given by $S_1=[\frac{13}{9}t^2+10t-9,\frac{1}{6}t^3]$ and $S_2=[t^2-1,\frac{1}{6}t^3]$. Moreover, any point $(\alpha,\beta)$ on the conic $X^2=10T^2 - 9$ leads to a point on its corresponding fiber $\E_\beta$ which is non-torsion and independent from the specialisation of $S_1$ and $S_2$. As a consequence, there are infinitely many rational fibers of $\E$ with rank at least 3. Under the parity conjecture, we deduce from \cite{Desjardins1} that the set of fibers on $\E$ with rank 4 is infinite.
\end{example}

\begin{remark}\label{rem:LS}
    Since the surfaces in Theorem \ref{thm:introresult}
 have rank 2 and that there is a conic on them, \cite{LS20} guaranties not only infinitely many points, but that the set of fibers with an elevated rank has Hilbert Property.
    The real interest of this note is thus the description of the sections $S_1,S_2$ and of the bisection $B$ on those surfaces and the way they are constructed. In a following paper, we construct many more elliptic surfaces with higher ranks jumps than \cite{LS20} for infinitely many fibers. We present an example of this sort in Section \ref{sec:examples} (Example \ref{ex:jump2}).
\end{remark}

\subsection{Acknowledgement}
The author is partially supported by an NSERC Discovery grant. She is grateful to Gary Walsh for his inspiring talk in the Quebec-Maine Number Theory Conference 2022 and to Cecilia Salgado and Rosa Winter for helpful discussions.

\section{Strategy}\label{sec:bisections}

As reported in the introduction, this note extends Example \ref{ex:Walsh1}. We start by briefly explain \cite{Walsh1} in order to understand the situation conceptually. The example was originally stated as follows:

\begin{theorem}\label{thm:Walsh}\cite{Walsh1}
Let $a,b,c$ be distinct integers. Then there is a computable constant $C=C(a,b,c)$ with the property that if $m>C$, then the rank of the curve $y^2=(x-a)(x-b)(x-c)+m^2$ is at least $2$.
\end{theorem}

A tentative sketch of the proof goes as follows:
Rewrite as $Y^2=X(X-A)(X-B)+m^2$, where $A=a-c$, $B=b-c$. Two sections are $S_1=(A,m)$ and $S_2=(B,m)$: they are non-torsion using Lutz-Nagell, and they are independent since $(A,m)+(B,m)=(0,m)$ is non-torsion as well.

An important observation is that both $S_1$ and $S_2$ come with a "twin" where the $y$-coordinate is $-m$. The elliptic surface $y^2=f(x)+m^2$ has the family of conics given by $m^2=f(x_o)-y^2$, and the two sections are components in that family: for $x_o=A$ and $x_o=B$ the family gives not a "bisection" like most other values of $x_o$, but rather the union of two sections (that are inverse from one another in the Mordell-Weil lattice). Moreover, they also correspond to two elements of 2-torsion of the rational points of the fiber at $m=0$:  indeed $\E_0(\Q)=(\Z/2\Z)^2$.

Those sections thus come from the degeneration of the family of conic passing through $(A,0)$ on $\E_0$ (resp. $(B,0)$) and $(0,1)$ on $\E_\infty$. It is natural to make the link with \cite{BulthuisVanLuijk}: on an rational elliptic surface without reducible fibers, given a fiber with a 2-torsion point, there exists a 1-dimensional family of genus 1 curves passing through this point. We use this idea to find Theorems \ref{thm:introresult}.

As mentioned before the elliptic surface described by the equation $y^2=(x-a)(x-b)(x-c)+m^2$ in Theorem \ref{thm:Walsh} is rational but has a reducible fiber of type $III^*$ at infinity. Samewise, $y^2=(x-a)(x-b)(x-c)+t^6$ in \cite[Theorem 1]{Walsh2} as a fiber of type $I_0^*$ at infinity. The fact that rational elliptic surfaces with those types of bad fibers are already covered by \cite[Proposition 3.2]{DCS}. This is actually why we care most about rational elliptic surfaces without reductive fibers, and in those cases the curves denoted by $S_1,S_2,B$ in Theorems \ref{thm:introresult} and \ref{thm:rank2withconic} are coming from bisections.

\subsection{Construct $S_1$ and $S_2$}

Let $\mathcal{E}$ be an elliptic surface given by the equation $$\E: y^2=x^3+(e_2t^2+e_1t+e_0)x^2+(f_4t^4+\dots+f_0)x+(g_6t^6+\dots+g_0)$$

In order to extend Theorem \ref{thm:Walsh}, we need more surfaces with two sections and a bisection of genus $0$. 

Thus set $S_1=[qt^2+rt+s,y_3t^3+y_0]$, $S_2=[ut^2+vt+w,y_2t^3+y_1]$: the parameters $q,r,s,u,v,w\in\Q$ need to be such that $S_1$ and $S_2$ are on $\E$ for any value of $t$. This leads to $y_1=y_0=0$ and to twelve more polynomials that should be equal to $0$: seven given by the equation \ref{eq:section1} to \ref{eq:section6}, and seven more where $q,r,s$ are replaced by $u,v,w$.

\begin{equation}\label{eq:section1}
    g_6+f_4q+e_2q^2+q^3-y_3^2\end{equation}
    \begin{equation}g_5+f_3q+f_4r+e_1q^2+2e_2qr
    +3q^2r\end{equation}
    \begin{equation}g_4+f_2q+f_3r+f_4s+ 2e_1qr+2e_2qs+e_2r^2+2q^2s-q^2w+3qr^2\end{equation}\begin{equation}g_3+ f_1q+f_2r+f_3s+2e_1qs+e_1r^2+2e_2sr+4qsr-2qwr +r^3-2y_0y_3\end{equation}
    \begin{equation}g_2+f_1r+
    f_2s+2e_1sr+e_2s^2+qs^2-qsw+2sr^2-wr^2
    \end{equation}
    \begin{equation}\label{eq:section6}g_1+f_1s
    +e_1s^2+s^2r-swr
\end{equation}

Moreover, we want a bisection of genus $0$. Suppose that $x=at^2+bt+c$, and that the bisection is given by the conic $(-s-w)-sw$. Then the seven following polynomials need to be 0 as well.

\begin{equation}g_6+f_4a+e_2a^2+a^3\end{equation} \begin{equation}g_5+f_3a+ f_4b +2e_1a^2+2e_2ab+3a^2b\end{equation} 
    \begin{equation}g_4+f_2a+f_3b+f_4c+2e_1ab+2e_2ac+e_2b^2-a^2s-a^2w+3a^2c + 3ab^2+s+w\end{equation}
    \begin{equation}g_3+f_1a+f_2b+f_3c+2e_1ac+e_1b^2+2e_2bc-2abs-2abw+6abc+b^3 \end{equation}
    \begin{equation}g_2+f_1b+f_2c+2e_1bc+e_2c^2+asw-2asc-2awc+3ac^2-b^2s-b^2w+3b^2c + sw\end{equation} \begin{equation}g_1+f_1c+e_1c^2+bsw-2bsc-2bwc+3bc^2\end{equation}\begin{equation}c(c-s)(c-w)\end{equation}

In order to have a similar situation as in Walsh, we set $c=0$ so that $B$ does not pass through the same points as $S_1$ and $S_2$ on the fiber $\E_0$.

\subsection{The sections $S_1$, $S_2$ are non-torsion}

\begin{lemma}\label{lemma:ntsection}
The points on $B_0$ are non-torsion on their fiber of $\mathcal{D}$ except for the 2-torsion point it was constructed from and potentially those at values of $t$ that divides $s^4-2s^3w+s^2w^2$.
\end{lemma}

\begin{proof}
We have $x(2C_+(p))=\frac{N}{D}$ where
$$N=20vt^7 + (-4s + 4w)t^6 - 16v^4t^4 - 16wv^3t^3 + 16s^2v^2t^2 +
    (-8s^3v + 8s^2wv)t + s^4 - 2s^3w + s^2w^2$$
   and 
$$D=4t(t^5 - 8v^3t^2 + (8sv^2 - 4wv^2)t + (-2s^2v + 2swv))$$
By Lutz-Nagell theorem if a point $P$ is torsion on an elliptic curve, then $x(2C_+(p))$ an integer, then it is not torsion.
Observe that $N$ has the form $4tN_1(t)+s^4-2s^2v+2swv$ for some $N_1$. Thus if $t\nmid s^4-2s^2v+2swv$, then $x(2C_+(p))$ is not integer and thus $C_+(p)$ is non-torsion.
\end{proof}

\subsection{Independance of $S_1$ and $S_2$ in the Mordell-Weil lattice}\label{sec:indep}

In generality, we can verify whether two sections of an elliptic surface are multiple of one another by computing their Gram matrix. 

Since $C_{1,+}$ and $C_{2,+}$ are sections, their self intersection numbers are both $-1$, and the Gram matrix is one the following 2:
\begin{equation*}
Gram(C_{1,+},C_{2,+}) = 
\begin{pmatrix}
-1 &  0 \\
0 & -1
\end{pmatrix}
\end{equation*}
that both have rank 2.  

Indeed, suppose they intersect, then $x(S_1)-x(S_2)$ and $y(S_1)-y(S_2)$ are equal to $0$: this simply mean that $(s-w)$ should be $0$, but we assumed it is not the case.

This proves the following lemma:

\begin{lemma}\label{lemma:ind}
    The generic rank of the elliptic surface denoted $\mathcal{D}$ in Theorem \ref{thm:introresult} is at least two. 
\end{lemma}

\section{Proof and an extension} \label{sec:proofs}

\begin{theorem}\label{thm:main1}(= Theorem \ref{thm:introresult}) Let $s,w\in\Q$ be distinct numbers such that the conic $Y^2+(s+w)T^2=sw$ is solvable. Let $(t,m)$ be a solution. Then for a given $v\in\Q$, there is only finitely many exceptions to the fact that the elliptic curve given by the Weierstrass equation
    $$\mathcal{D}:y^2=x^3-(3vt+s+w) x^2+ 
(3v^2t^2 -2v(s+w)t+sw)x
+ t^6 - v^3t^3 -v^2(s+w)t^2+ swvt$$
has rank at least 3, with the three following independant non-torsion points $[-t^2+vt,tm]$, $[vt+s,t^3]$, $[vt+w,t^3]$.
\end{theorem}

\begin{proof}
We can easily check that those are indeed non-torsion sections of $\E$ in a similar way as Lemma \ref{lemma:ntsection}. The fact that the three points are independent from one another can be proved in a similar way to Lemma \ref{lemma:ind}.
\end{proof}

All construction of the three independant points in Theorem \ref{thm:main1} can be extended to the following Theorem \ref{thm:encoreplusmain}:

\begin{theorem}\label{thm:encoreplusmain}
Let $w,s\in\Z$ be non-zero integers not equal to one another and such that one is equal to a square times the other. Suppose them also such that the conic $C:Y^2=(-s-w)X^2-sw$ has a solution.

Given $b,v$ such that the equations are well defined, let $$\alpha_1(t)=\Big(f_4t^4 + \frac{2absw + 2abw^2 + 2aw^2v + b^3s - b^2sv + 2bsw - bsv^2 + 2bw^2 + sv^3 + 2w^2v}{w^2}t^3$$$$ + \frac{-e_2w^2 + asw + b^2s +
    b^2w + 2bwv + sw - sv^2}{w}t^2 + (2bs + 2bw)t + sw),$$ 
    $$\alpha_2(t)=\Big(e_2t^2 + \frac{-b(s+ 2w) + v(s - w)}{w}t - s - w\Big),$$  and let $\mathcal{D}$ be the surface given by the equation 

$$\mathcal{D}:y^2=x^3+\alpha_2(t)x^2+
\alpha_1(t)x+
    g_6t^6 $$$$+ \frac{(-a^2swv - a^2w^2v - ab^2sv + 2absv^2 - 2aswv -asv^3 - 2aw^2v - b^2sv + 2bsv^2 - swv - sv^3 - w^2v)}{w^2}t^5$$$$ +\frac{(-sw^2-w^3)(a+1)^2+sv(b-v)^2b+b(2v(a+1)w^2)+b(+swb(a+1))}{w^2}t^4$$$$
    + \frac{e_2bw^2 - absw - b^2sv - b^2wv - bsw + bsv^2}{w}t^3 + (-(a+1)sw
    - b^2(s+w))t^2 - bswt$$
where $y_2=\pm y_3;$, $g_6=y_3^2$ and $f_4,e_2,a\in\Q$ are parameters depending on the values of $u,s,w,y_3$ as in equations \ref{eq:g6},\ref{eq:f4},\ref{eq:e2}. Suppose that we additionally have one of the following: 
\begin{itemize}
    \item $v=0$
    \item $v=2b$
    \item $v^2 - 2vb + 2b^2 + (2l^2w + 2w)/l^2=0$ (where $l$ is).
\end{itemize}

Then for every fibers except finitely many, the three following points (that are rational except possibly for $B_1(t)$ which is defined over $\overline{\mathbb{Q}}$) are non-torsion and independent from one another:
$$B_1(t):=[at^2+bt,t\sqrt{(-s-w)t^2-sw}]
\quad S_1(t)=\left[vt + w,y_3t^3\right]$$$$
S_2(t)=\left[\frac{b^2s^2-b^2sw-2bs^2v+2bswv+s^2w+s^2v^2-swv^2-w^3}{2sw^2}t^2+\frac{b(s+w)-sv}{w}t+s,y_3t^3\right]$$
Thus there are infinitely many fibers with rank 3 or more.

\end{theorem}

\begin{equation}\label{eq:g6}
g_6=\frac{1}{4s^2w^3}(b^4us^3-b^4us^2w-2b^2u^2s^3w+4b^2u^2s^2w^2+
2b^2us^2w^2-2b^2usw^3-4u^3s^2w^3-2u^2s^3w^2
\end{equation}
\begin{equation*}
-2u^2s^2w^3-us^3w^2+us^2w^3 + 4s^2w^3y_3^2+4u^2sw^4+usw^4-uw^5)\end{equation*}
\begin{equation}\label{eq:f4}f_4=-\frac{(-aqw^2- auw^2-asw-aw^2-b^2s+2bsv-quw^2 + usw- uw^2 - sw - sv^2 - w^2)}{w^2}\end{equation}

    \begin{equation}\label{eq:e2}e_2=-\frac{as^2-asw+qs^2- qsw + us^2 - usw + s^2 - sw}{s(s-w)}\end{equation} and \begin{equation}\label{eq:a}a=-\frac{b^3s-3b^2sv-2busw+ bsw+3bsv^2+bw^2+2uswv-swv-sv^3-w^2v}{2sw(b-v)}\end{equation} $$y_3=\frac{sw+sb^2+w^2}{2w^2l}$$

\begin{example}\label{ex:julie1}
    Let us take the following elliptic surface given by
    $$\E:y^2=x^3+(-6t + 5)x^2+\left(-\frac{37}{192}t^4 - \frac{1}{4}t^3 + \frac{53}{12}t^2 - 10t + 4\right)x$$$$+\frac{377}{6912}t^6 + \frac{15}{16}t^5 + \frac{749}{576}t^4 + \frac{7}{12}t^3 + \frac{19}{6}t^2 - 4t$$
This surface is isomorphic to $\mathcal{D}$ where we set $w=-1$, $s=4$, $a=-13/24$, $v=b=1$ and $u=-5/12$.

Thus $\E$ has the bisection $B=[-\frac{13}{24}t^2 + t,t\sqrt{5t
^2-4}]$
, and two rational sections $S_1=[-\frac{1}{24}t^2+5t-4,\frac{1}{4}t^3]$ and 
$S_2=[-\frac{5}{12}t^2 - 1,\frac{1}{4}t^3]$.

    Notice that the conic $X^2=5T^2-4$ has infinitely many solutions. For any of them, say $(x,t)$, we have that $\E_t$ has rank 3 or up.
\end{example}

\begin{remark}\label{rm:iso}
Observe that the discriminant and the c-invariants of the following elliptic surface $\mathcal{D}$ does not actually depend of $u$. Thus in particular, the conclusion in $\E$ extends to any elliptic surface with the same parameters except $u\in\Q$ simply because they are isomorphic: 
    $$\E:y^2=x^3+\left(\left(b^2-3u-\frac{9}{4}\right)t^2-6bt+5\right)x^2+\Big(\left(-\frac{3}{4}b^4-2b^2u +\frac{3}{8}b^2+3u^2+\frac{9}{2}u+\frac{45}{64}\right)t^4$$$$+\left(b^3+12bu+\frac{15}{4}b\right)t^3+\left(4b^2-10u- \frac{15}{4}\right)t^2-10bt + 4)x$$$$+\left(\frac{3}{4}b^4u+b^4+b^2u^2 - \frac{3}{8}b^2u-\frac{5}{2}b^2 - u^3 - \frac{9}{4}u^2-\frac{45}{64}u+\frac{25}{16}\right)t^6+\left(-b^3u- 6bu^2-\frac{15}{4}bu\right)t^5$$$$+\left(-\frac{3}{4}b^4-4b^2u+
\frac{3}{8}b^2+5u^2+\frac{15}{4}u+\frac{45}{64})t^4+(b^3+ 10bu+\frac{15}{4}b\right)t^3 + \left(3b^2
    -4u-\frac{3}{2}\right)t^2 - 4bt.$$
    Indeed, for these surfaces as for the surface in Remark \ref{ex:julie1} we have $\Delta(\E_t)=-\frac{624375}{1024}t^{12} - \frac{206145}{64}t^{10} + \frac{99279}{16}t^8 - 2690t^6 + 252t^4 -
    2880t^2 + 2304$ and $c_4(\E_t)=\frac{189}{4}t^4 - 180t^2 + 208$ both of these polynomials not depending on the variable $u$.

The sections on $\E$ have equation $S_1=[vt-1,
(-b^2+\frac{5}{4})t^3$, $S_2=[(-\frac{3}{2}b^2+3bv - \frac{3}{2}v^2 + \frac{15}{8})t^2+(5b -4v)t-4,
(-b^2 + \frac{5}{4})t^3
]$ and the bisection is given by $B=[(\frac{1}{2}b^2 - bv + \frac{1}{2}v^2 - \frac{5}{8})t^2 + bt,]$

\end{remark}

\section{Search for high rank elliptic curves}

\begin{remark}
We consider the elliptic surface given by the equation
$$\mathcal{L}_{v,s}:y^2=x^3-3vtx^2+
(3v^2t^2 - w^2)x+ t^6 -v^3t^3 + w^2vt$$
which is found from $\mathcal{D}$ by restricting $s=-w$.

Computing with magma we find the following rank for the fibers of $\E=\mathcal{L}_{1,1}$ (when $v=1,w=-1$) as $t$ varies from $1$ to $10$:
\begin{align}
\text{equation}&\quad\text{ rank}\\
y^2 = x^3 - 3x^2 + 2x + 1 &\quad1\\
y^2 = x^3 - 6x^2 + 11x + 58&\quad3\\
y^2 = x^3 - 9x^2 + 26x + 705&\quad4\\
y^2 = x^3 - 12x^2 + 47x + 4036&\quad3\\
y^2 = x^3 - 15x^2 + 74x + 15505&\quad4\\
y^2 = x^3 - 18x^2 + 107x + 46446&\quad3\\
y^2 = x^3 - 21x^2 + 146x + 117313&\quad4\\
y^2 = x^3 - 24x^2 + 191x + 261640&\quad5\\
y^2 = x^3 - 27x^2 + 242x + 530721&\quad5\\
y^2 = x^3 - 30x^2 + 299x + 999010&\quad4\\
\end{align}

On $\mathcal{E}_1$, we have $S_2(1)=2S_1(1)=-S_1(1)$. It is one of the finitely many exceptions allowed by Silverman's theorem where the rank of the specialization is strictly lower than the generic rank.
\end{remark}

Like in \cite{Walsh2}, we can also play around with Pell's equation. If we start with $a,b$ non-zero, we can write $b=t^2+a-1$ so that $a+b=t^2-1$, generating K3 surfaces or surfaces of general type, whose fibers are likely to have higher rank, although those are unfortunately hard to compute.

Return to Example \ref{ex:rank2withconic}. If we hope to proceed similarly to \cite{Walsh1}, we quickly hit a wall. Since the conic $C$ is soluble whenever $s=1$, set $w=t^2-2$ so that $s+w=t^2-1$. Take the simple solution $(1,1)$ to $C_{-9,-1}$, and consider the associated real number $\alpha=1+\sqrt{10}$. We can create more solution to $C_{-9,-1}$ by multiplying $\alpha$ by powers of the fundamental unit, and thus search more easily for fibers of high rank in $\E$:

\begin{equation*}
\begin{aligned}
    k\text{, } \alpha^k=x_k+t_k\sqrt{10}&\qquad \qquad \E_{t_k}& \text{rank of }\mathcal{E}_{t_k}\\
    0\text{, }\qquad \alpha^0=1+\sqrt{10}&\qquad y^2=x^3-\frac{97}{18}x^2+\frac{583}{81}x+\frac{1}{36}&3\\
    1\text{, }\quad \alpha^1=79+25\sqrt{10}&\qquad y^2 = x^3 - \frac{54145}{18}x^2 + \frac{235406479}{81}x -
\frac{94870014925}{108}&
\end{aligned}
\end{equation*}
Unfortunately, the height of the coefficients and the points grows very quickly and it is hard for the computer to compute even the first iteration. Thus, this paper does not seem to help the search for higher rank jumps with the same technique as in \cite{Walsh1} or \cite{Walsh2}.

If instead, we study the rank among the fibers for small values of $t$, we obtain:

\begin{equation*}
\begin{aligned}
  \text{t, }&\qquad\E_t& \text{Rank of $\E_t$}\\
-4,& \qquad y^2 = x^3 - \frac{146}{9}x^2 - \frac{55367}{81}x + \frac{287332}{27} &
4 \\
-3,& \qquad y^2 = x^3 + {7}{2}x^2 - 469x + \frac{12145}{4} &
3 \\
-2,& \qquad y^2 = x^3 + \frac{130}{9}x^2 - \frac{15107}{81}x + \frac{10916}{27} &
2 \\
-1,& \qquad y^2 = x^3 + \frac{299}{18}x^2 - \frac{1343}{81}x + \frac{1}{36} &
4 \\
0,& \qquad y^2 = x^3 + 10x^2 + 9x &
0 \\
1,& \qquad y^2 = x^3 - \frac{97}{18}x^2 + \frac{583}{81}x + \frac{1}{36} &
3 \\
2,& \qquad y^2 = x^3 - \frac{266}{9}x^2 + \frac{19741}{81}x - \frac{13240}{27} &
3 \\
3,& \qquad y^2 = x^3 - \frac{125}{2}x^2 + 1133x - \frac{22223}{4} &
3 \\
4,& \qquad y^2 = x^3 - \frac{938}{9}x^2 + \frac{262297}{81}x - \frac{766388}{27} &
3 \end{aligned}
\end{equation*}

Another trick to get an elliptic surface of high rank is to set as previously $s=1$, $w=t^2-1$, but consider $t$ as a parameter and $v$ as the variable. Observe however that the resulting elliptic surface $$y^2=x^3-3vtx^2+(3v^2t^2-(t^2-1)^2)x+t^6-v^3t^3+(t^2-1)^2vt$$ now has a reducible fiber at the place at infinity. 

\section{Similar examples}\label{sec:examples}

We manage to find more examples extending in different ways Theorem \ref{thm:introresult}, like the ones presented in this section. With our technique, we obtain even more interesting rank jump phenomenons or small generic ranks. This method is based on the construction of genus $0$ curves on the surfaces, from bisections. More details will be available soon in a further paper. 

There are more elliptic surfaces with two sections and a genus 0 conic than those in Theorem \ref{thm:introresult}, for example:

\begin{example}\label{ex:rank2withconic}
Let $\mathcal{H}$ be the elliptic surface given by the equation $$\mathcal{H}:y^2=x^3+\left(-\frac{79}{18}t^2 - 11t + 10\right)x^2+
\left(\frac{502}{81}t^4 + \frac{287}{9}t^3 - \frac{179}{9}t^2 - 20t + 9\right)x$$$$+
\left(-\frac{901}{324}t^6 - \frac{188}{9}t^5 + \frac{835}{81}t^4 + \frac{269}{9}t^3 - \frac{15}{2}t^2 - 9t\right)$$
Then the generic rank is 2, with generating sections given by $S_1=[\frac{13}{9}t^2+10t-9,\frac{1}{6}t^3]$ and $S_2=[t^2-1,\frac{1}{6}t^3]$. Moreover, any point $(\alpha,\beta)$ on the conic $X^2=10T^2 - 9$ leads to a point on its corresponding fiber $\E_\beta$ which is non-torsion and independent from the specialisation of $S_1$ and $S_2$. As a consequence, there are infinitely many rational fibers of $\E$ with rank at least 3. Under the parity conjecture, we deduce from \cite{Desjardins1} that the set of fibers on $\E$ with rank 4 is infinite.

\end{example}

\begin{remark}
The base change of $\mathcal{H}$ by the conic $X^2=10T^2-9$ gives an elliptic surface that is not rational: it is a K$3$-surface. 
Nagao's formula is only conjectural and we cannot apply \cite{BBD} to compute the generic rank on those covers. However we know thanks to the properties of $\mathcal{H}$ that the generic rank of the base change is at least 3. 
\end{remark}

\begin{remark}
    The surfaces in Theorems \ref{thm:introresult} 
    have no reducible bad fiber: For the surface $\mathcal{D}$ we have
    $c_4(t)=16(s^2 - sw + w^2)$ and
$c_6(t)=32\left(-27t^6 + 2s^3 - 3s^2w - 3sw^2 + 2w^3\right)$.

Both's surfaces discriminant are irreducible polynomials of degree 12 (except for particular values of $v$, for instance $v=0$.). This means that generically $\mathcal{D}$ and $\mathcal{H}$ are isomorphic to the blowups of del Pezzo surfaces of degree 1 at their anticanonical points.
\end{remark}

Example \ref{ex:rank2withconic} is, like Theorem \ref{thm:introresult}, already contained in \cite{LS20}, although being more explicit. However Example \ref{ex:jump2} is not contained neither in \cite{LS20} nor in \cite{DCS}. Example \ref{ex:rank2withconic} is a subcase of the following Theorem \ref{thm:rank2withconic}:

\begin{theorem}\label{thm:rank2withconic}
    Let $\E:y^2=x^3+ex^2+fx+g$ with $e,f,g\in\Q[t]$ described as follows.
    
    $$e=\frac{-1/2b^2l^4 + b^2l^2 - 3ul^2w - 1/2l^4w - 1/2l^2w +
    w}{l^2w}t^2 + -b(l^2 +2)t - w(l^2+1)$$
    
$$f= \frac{\frac{1}{4}(-l^2+1)b^4l^4+(ul^4-(2u-1/2)l^2
    + \frac{1}{2})b^2l^2w + ((3u^2+u-\frac{1}{4})l^4+(u+\frac{1}{4})l^6-(2u-\frac{1}{4})l^2+\frac{1}{4})w^2}{l^4w^2}t^4$$$$ +
    (-b^3l^2 + 2bul^4w + 4bul^2w + bl^4w - bw)/(l^2w)t^3 $$$$+
    (b^2l^4 + 2ul^4w + 2ul^2w + l^4w - w)/l^2t^2 + 2bw(l^2+1)t + l^2w^2$$
    
$$g= \frac{1}{4w^2}\Big(l^2(b^4(u+1)-2b^2u^2w  + 2b^2w-2u^2w^2 +w^2-uw^2)$$$$+(4b^2u^2w -b^4u+2b^2uw +
    2b^2w -4u^3w^2-2u^2w^2+uw^2  + 2w^2)$$$$+
    \frac{4u^2w^2+uw^2
    - 2b^2uw+w^2}{l^2}-\frac{uw^2}{l^4}\Big)t^6$$
    $$ + \frac{b^3ul^2 -
    bu^2l^4w - 2bu^2l^2w - bul^4w + buw}{l^2w}t^5 $$$$+
    \frac{\frac{1}{4}b^4l^4(l^2 - 1) - (ul^4 +\frac{1}{2}l^2 - \frac{1}{2})b^2l^2w +w^2(-u^2l^4(l^2-1)-ul^2(l^4+1)-
    \frac{1}{4}(l^6+l^4+l^2-1)}{l^4w}t^4 $$$$+
    \frac{(b^3l^2 - 2bul^4w - 2bul^2w - bl^4w + bw)}{l^2}t^3 +
    \frac{1}{2}(-b^2l^2w - 2b^2w - 2ul^2w^2-l^2w^2+w^2)t^2 -
    bl^2w^2t$$

    Then we have the two independent sections $$B_1=\left[\frac{b^2l^4-b^2l^2 + 2ul^2w +l^4w -w}{2l^2w}t^2 +
    (bl^2 + b)t + l^2w,
\frac{(b^2l^2 + l^2w + w)}{2lw}t^3\right]$$ and $$B_2=
\left[ut^2 + w
,\frac{b^2l^2+l^2w+w}{2lw}t^3\right].$$ Moreover, any point $(X_0,T_0)$ on the conic $C_{l,w}:X^2=(-l^2w - w)T^2 - l^2w^2$ lead to a point on $\E_T$ which is non-torsion and does not depend on $B_1 (T_0)$ and $B_2(T_0)$. As a consequence, given that $C_{l,w}$ admits points, there are infinitely many fibers of $\E$ with rank at least 3.
\end{theorem}

By \cite{LS20} using that conic leads by base change to a elliptic surface on which that conic is transformed into a section. We can find explicit elliptic surfaces on which the conic is already a pair of rational sections (Theorem \ref{thm:introresult}), or some where like in Walsh's example the conic is not a union of two rational sections (Theorem \ref{thm:rank2withconic}). 

\begin{example}\label{ex:jump2}
    Let $\E$ be the rational elliptic surface given by the equation $$\E:y^2=x^3+ 2x^2+(
-t^4+2t^3+t^2-2t-1)x+
4t^3+3t^2-4t-2$$
    
     $B_1=[-t^2+t+1,t\sqrt{2t+1}]$ and $B_2=[t^2-t-1,t\sqrt{2t+1}]$ are non-torsion and independant from one another in the Mordell-Weil group.
    
    Therefore, there are infinitely many fibers of $\E$ where the rank is 2. 
    By \cite{Desjardins1}, there are infinitely many fibers of $\E$ where the rank is 3.
\end{example}

\begin{remark}
    Example \ref{ex:jump2} is hard to construct because we not only need two bisections of genus 0, but also that they have rational points at the same values of $t$. In that example, the two values in $y$ are the same. We still have not found an example with three bisections that have rational points for the same values of $t$.
\end{remark}

\bibliographystyle{amsalpha}
\bibliography{DensityDP1}

\end{document}